\documentclass[english,10pt]{amsart}

\usepackage{amsmath, amssymb, amsthm, enumerate}
\usepackage[all]{xy}
\usepackage[activeacute]{babel}

\newtheorem{thm}{Theorem}[section]
\newtheorem{lem}[thm]{Lemma}
\newtheorem{prop}[thm]{Proposition}
\newtheorem{cor}[thm]{Corollary}

\newtheorem{defi}[thm]{Definition}

\newtheorem{rmk}[thm]{Remark}

\newcommand{\Hom}{\mathrm{Hom}}
\newcommand{\Gm}{\mathbb G _{m}}
\newcommand{\TspectraX}{Spt(\mathcal M)}
\newcommand{\stablehomotopyX}{\mathcal{SH}}

\newcommand{\neffstablehomotopyX}{\Sigma _{T}^{n}\stablehomotopyX^{\mathit{eff}}}
\newcommand{\qorthogonalTspectraX}{L_{<q}\TspectraX}

\newcommand{\nplusoneorthogonalTspectraX}{L_{<n+1}\TspectraX}
\newcommand{\qplusoneorthogonalTspectraX}{L_{<q+1}\TspectraX}
\newcommand{\qorthogonalstablehomotopyX}{L_{<q}\stablehomotopyX}
\newcommand{\qorthogonalX}{\stablehomotopyX ^{\perp}(q)}
\newcommand{\northogonalX}{\stablehomotopyX ^{\perp}(n)}

\newcommand{\nbiratTspectraX}{Spt (B_{n} \mathcal M)}

\newcommand{\stablenbiratTspectraX}{\stablehomotopyX (B_{n})}
\newcommand{\nwbiratTspectraX}{Spt (WB_{n} \mathcal M)}

\newcommand{\jwbiratTspectraX}{Spt (WB_{j} \mathcal M)}

\newcommand{\stablenwbiratTspectraX}{\stablehomotopyX (WB_{n})}

\newcommand{\kdim}{\mathrm{dim}}
\newcommand{\blowup}{\mathcal B \ell}

\numberwithin{equation}{section}

\begin{document}


\title{Birational Motivic Homotopy Theories and the Slice Filtration}


\author{Pablo Pelaez}
\address{Department of Mathematics, Rutgers University, U.S.A.}
\email{pablo.pelaez@rutgers.edu}





\begin{abstract}
	We show that there is an equivalence
	of categories between the orthogonal components for the slice filtration and
	the birational motivic stable homotopy categories which are constructed in this
	paper.  Relying on this equivalence, we
	are able to describe the slices for projective
	spaces (including $\mathbb P ^{\infty}$), Thom spaces and blow ups.
\end{abstract}


\maketitle


\begin{section}{Definitions and Notation}
		\label{Introd}
			
	Our main result, theorem \ref{thm.orthogonal=wbirationalTspectra}, shows that there is an equivalence
	of categories between the orthogonal components for the slice filtration (see definition \ref{def.orthogonality})
	and the weakly birational motivic stable homotopy categories which are constructed in this
	paper (see definition \ref{def.localizationAmod-weaklyBirat}).  Relying on this equivalence; we
	are able to describe over an arbitrary base scheme (see theorems \ref{thm.slices.proj.space}, \ref{thm.slices-thom} and
	\ref{thm.blowups}) the slices for projective
	spaces (including $\mathbb P ^{\infty}$), Thom spaces and blow ups.
	We also construct the birational motivic stable homotopy categories (see definition \ref{def.localizationAmod-Birat}), which are
	a natural generalization of the weakly birational motivic stable homotopy categories, and show 
	(see proposition \ref{prop.genericsmoothness-Quillenequiv}) that there exists a Quillen equivalence between 
	them when the base scheme is a perfect field.
	Our approach was inspired by the
	work of Kahn-Sujatha \cite{K-theory/0596} on birational motives, where  the existence
	of a connection between the layers of the slice filtration and birational invariants
	is explicitly suggested.  Furthermore, this approach allows to obtain analogues for the slice filtration
	in the unstable setting (see remark \ref{rmk.unstable.case}).
		
	In this paper $X$ will denote a Noetherian separated base scheme of finite Krull dimension,
	$Sch_{X}$ the category of schemes of finite type over $X$ and $Sm_{X}$ the full
	subcategory of $Sch_{X}$ consisting of smooth schemes over $X$ regarded
	as a site with the Nisnevich topology.	  All the maps between schemes will be considered over
	the base $X$.  Given $Y\in Sch_{X}$, all the closed subsets $Z$ of $Y$ will be considered
	as closed subschemes with the reduced structure.
	
	Let  $\mathcal M$ be the category of pointed simplicial presheaves in $Sm_{X}$
	equipped with the motivic Quillen model structure \cite{MR0223432} constructed by 
	Morel-Voevodsky \cite[p. 86 Thm. 3.2]{MR1813224}, taking the affine line $\mathbb A _{X}^{1}$
	as interval.  Given a map $f:Y\rightarrow W$ in $Sm_{X}$, we will abuse notation and denote
	by $f$ the induced map $f:Y_{+}\rightarrow W_{+}$ in $\mathcal M$ between the corresponding pointed
	simplicial presheaves represented by $Y$ and $W$ respectively.
	
	We define $T$ in $\mathcal M$ to be the pointed simplicial presheaf represented by 
	$S^{1}\wedge \mathbb G_{m}$, where $\mathbb G_{m}$ is the multiplicative group 
	$\mathbb A^{1}_{X}-\{ 0 \}$ pointed by $1$, and $S^{1}$ denotes the simplicial circle.
	Given an arbitrary integer $r\geq 1$,
	$S^{r}$ (respectively $\Gm ^{r}$) will denote the iterated smash product 
	$S^{1}\wedge \cdots \wedge S^{1}$ (respectively $\Gm \wedge \cdots \wedge \Gm$)
	with $r$-factors;
	$S^{0}=\Gm ^{0}$ will be by definition equal to the pointed simplicial presheaf $X_{+}$ represented by the base
	scheme $X$.
	
	Let $\TspectraX$ denote Jardine's category of symmetric $T$-spectra on 
	$\mathcal M$ equipped with the motivic model structure defined in 
	\cite[Thm. 4.15]{MR1787949} and let $\stablehomotopyX$ denote its homotopy category, 
	which is triangulated.  We will follow Jardine's notation \cite[p. 506-507]{MR1787949}
	where $F_{n}$ denotes the left adjoint to the $n$-evaluation functor
		\[ \xymatrix@R=0.5pt{\TspectraX \ar[r]^-{ev_{n}}& \mathcal M \\
					(X^{m})_{m\geq 0} \ar@{|->}[r]& X^{n}}
		\]
	Notice that $F_{0}(A)$ is just the usual infinite suspension spectrum $\Sigma _{T}^{\infty}A$.
	
	For every integer $q\in \mathbb Z$, we consider the following family of symmetric $T$-spectra
		\begin{align}
				\label{eqn.Cqeff}
		 	C^{q}_{\mathit{eff}}=\{ F_{n}(S^{r}\wedge \mathbb G _{m}^{s}\wedge U_{+}) 
				\mid n,r,s \geq 0; s-n\geq q; U\in Sm_{X}\}
		\end{align}
	where $U_{+}$ denotes the simplicial presheaf represented by $U$ with a disjoint base point.
	Let $\Sigma _{T}^{q}\stablehomotopyX^{\mathit{eff}}$ denote the smallest full triangulated subcategory of 
	$\stablehomotopyX$ which contains
	$C^{q}_{\mathit{eff}}$ and is closed under arbitrary coproducts.
	Voevodsky \cite{MR1977582} defines the slice filtration in $\stablehomotopyX$
	to be the following family of 
	triangulated subcategories
		\[ \cdots \subseteq \Sigma _{T}^{q+1}\stablehomotopyX^{\mathit{eff}} 
			\subseteq \Sigma _{T}^{q}\stablehomotopyX^{\mathit{eff}}
			\subseteq \Sigma _{T}^{q-1}\stablehomotopyX^{\mathit{eff}} \subseteq \cdots
		\]
	It follows from the work of Neeman \cite{MR1308405}, \cite{MR1812507} that the inclusion
		\[ i_{q}:\Sigma _{T}^{q}\stablehomotopyX^{\mathit{eff}}\rightarrow \stablehomotopyX
		\]
	has a right adjoint $r_{q}:\stablehomotopyX \rightarrow \Sigma _{T}^{q}\stablehomotopyX^{\mathit{eff}}$, 
	and that the following functors
		\begin{align*}
				f_{q}: & \stablehomotopyX \rightarrow \stablehomotopyX \\
				s_{<q}: & \stablehomotopyX \rightarrow \stablehomotopyX \\
				s_{q}: & \stablehomotopyX \rightarrow \stablehomotopyX
		\end{align*}
	are triangulated, where $f_{q}$ is defined as the composition $i_{q}\circ r_{q}$; and $s_{<q}, s_{q}$ 
	are characterized by the fact that for every $E\in \stablehomotopyX$, we have
	distinguished triangles in $\stablehomotopyX$:
		\[ \xymatrix{f_{q}E \ar[r]^-{\theta _{q}^{E}}& 
							E \ar[r]^-{\pi _{<q}^{E}}& s_{<q}E \ar[r]& S^{1}\wedge f_{q}E\\
							f_{q+1}E \ar[r]^-{\rho _{q}^{E}}& f_{q}E 
							\ar[r]^-{\pi _{q}^{E}}& s_{q}E \ar[r]& S^{1}\wedge f_{q+1}E}
		\]
	We will refer to $f_{q}E$ as the \emph{$(q-1)$-connective cover} of $E$, to $s_{<q}E$ as the 
	\emph{$q$-orthogonal component} of $E$, and to $s_{q}E$ as the \emph{$q$-slice} of $E$.
	It follows directly from the definition that $s_{<q+1}E, s_{q}E$ satisfy 
	that for every symmetric $T$-spectrum $K$ in $\Sigma _{T}^{q+1}\stablehomotopyX^{\mathit{eff}}$:
		\[ \Hom _{\stablehomotopyX}(K,s_{<q+1}E)=\Hom _{\stablehomotopyX}(K,s_{q}E)=0
		\]
		
\begin{defi}
		\label{def.orthogonality}
	Let $E\in \TspectraX$ be a symmetric $T$-spectrum.		
	We will say that $E$ is \emph{$n$-orthogonal},
	if for all $K\in \neffstablehomotopyX$
		$$\Hom _{\stablehomotopyX}(K,E)=0$$
	Let $\northogonalX$ denote the full subcategory of $\stablehomotopyX$ consisting
	of the $n$-orthogonal objects.
\end{defi}	
	
	The slice filtration admits an alternative definition in terms of (left and right) Bousfield localization
	of $\TspectraX$ \cite{MR2576905, MR2807904}.  The Bousfield localizations are constructed following
	Hirschhorn's approach \cite{MR1944041}.  In order to be able to apply Hirschhorn's techniques, it is necessary
	to know that $\TspectraX$ is \emph{cellular} \cite[Def. 12.1.1]{MR1944041} and \emph{proper} 
	\cite[Def. 13.1.1]{MR1944041}.  
	
\begin{thm}
		\label{Tspectra-cellular}
	The Quillen model category $\TspectraX$ is:
		\begin{enumerate}  
			\item  \label{Tspectra-cellular.a}  \emph{cellular} (see \cite{MR1860878}, 
							\cite[Cor. 1.6]{MR2197578} or \cite[Thm. 2.7.4]{MR2807904}).
			\item  \label{Tspectra-celllular.b}  \emph{proper} (see \cite[Thm. 4.15]{MR1787949}).
		\end{enumerate}
\end{thm}

	For details and definitions about Bousfield localization we refer the reader to Hirschhorn's book \cite{MR1944041}.
	Let us just mention the following theorem of Hirschhorn, which guarantees the existence of left and right Bousfield localizations.
	
\begin{thm}[{see \cite[Thms. 4.1.1 and 5.1.1]{MR1944041}}]
		\label{Hirsch-Bousloc}
	Let $\mathcal A$ be a Quillen model category which is cellular and proper.  Let $L$ be a set of maps in $\mathcal A$ and
	let $K$ be a set of objects in $\mathcal A$.
	Then:
	\begin{enumerate}
		\item	The left Bousfield localization of $\mathcal A$ with respect to $L$ exists.
		\item	The right Bousfield localization of $\mathcal A$ with respect to the class of $K$-colocal
				equivalences exists.
	\end{enumerate}
\end{thm}
	
	Now, we can describe the slice filtration in terms of suitable Bousfield localizations of $\TspectraX$.
	
\begin{thm}[{see \cite{MR2807904}}]
		\label{thm.modelstructures-slicefiltration}
	\begin{enumerate}
		\item	\label{thm.modelstructures-slicefiltration.a}
				Let $R_{C^{q}_{\mathit{eff}}}\TspectraX$ be the
				right Bousfield localization of $\TspectraX$ with respect to the set of 
				objects $C^{q}_{\mathit{eff}}$ (see Eqn. \eqref{eqn.Cqeff}).  Then its
				homotopy category $R_{C^{q}_{\mathit{eff}}}\stablehomotopyX$
				is triangulated and naturally equivalent to
				$\Sigma _{T}^{q}\stablehomotopyX^{\mathit{eff}}$.
				Moreover, the functor $f_{q}$ is canonically isomorphic to the following
				composition of triangulated functors:
					$$\xymatrix{\stablehomotopyX \ar[r]^-{R}& R_{C^{q}_{\mathit{eff}}}\stablehomotopyX \ar[r]^-{C_{q}}&
					 \stablehomotopyX}
					$$
				where $R$ is a fibrant replacement functor in $\TspectraX$, and $C_{q}$ a cofibrant replacement
				functor in $R_{C^{q}_{\mathit{eff}}}\TspectraX$.
		\item	\label{thm.modelstructures-slicefiltration.b}
				Let $\qorthogonalTspectraX$ be the
				left Bousfield localization of $\TspectraX$ with respect to the set of 
				maps 
					\[ \{ F_{n}(S^{r}\wedge \Gm ^{s}\wedge U_{+})\rightarrow \ast | \;
							F_{n}(S^{r}\wedge \Gm ^{s}\wedge U_{+})\in C^{q}_{\mathit{eff}} \}			
					\]
				Then its homotopy category $\qorthogonalstablehomotopyX$
				is triangulated and naturally equivalent to
				$\qorthogonalX$.
				Moreover, the functor $s_{<q}$ is canonically isomorphic to the following
				composition of triangulated functors:
					$$\xymatrix{\stablehomotopyX \ar[r]^-{Q}& L_{<q}\stablehomotopyX \ar[r]^-{W_{q}}&
					 \stablehomotopyX}
					$$
				where $Q$ is a cofibrant replacement functor in $\TspectraX$, and $W_{q}$ a fibrant replacement
				functor in $\qorthogonalTspectraX$.
		\item \label{thm.modelstructures-slicefiltration.c}
				Let $S^{q}\TspectraX$ be the right
				Bousfield localization of $\qplusoneorthogonalTspectraX$ with respect to the set of objects
					\[	\{ F_{n}(S^{r}\wedge \mathbb G _{m}^{s}\wedge U_{+}) 
									\mid n,r,s \geq 0; s-n= q; U\in Sm_{X}\}
					\]
				Then its homotopy category 
				$S^{q}\stablehomotopyX$ is triangulated
				and the identity functor 
					\[	id:R_{C^{q}_{\mathit{eff}}}\TspectraX \rightarrow S^{q}\TspectraX
					\]
				is a left Quillen functor.  Moreover, the functor $s_{q}$ is canonically isomorphic
				to the following composition of triangulated functors:
					$$\xymatrix{\stablehomotopyX \ar[r]^-{R} & R_{C^{q}_{\mathit{eff}}}\stablehomotopyX \ar[r]^-{C_{q}}& 
									S^{q}\stablehomotopyX \ar[r]^-{W_{q+1}}&
									R_{C^{q}_{\mathit{eff}}}\stablehomotopyX \ar[r]^-{C_{q}}& \stablehomotopyX}
					$$
	\end{enumerate}
\end{thm}
\begin{proof}
	(\ref{thm.modelstructures-slicefiltration.a}) and (\ref{thm.modelstructures-slicefiltration.c})  
	follow directly from \cite[Thms. 3.3.9, 3.3.25, 3.3.50, 3.3.68]{MR2807904}.
	On the other hand, (\ref{thm.modelstructures-slicefiltration.b}) follows from
	proposition 3.2.27(3) together with theorem 3.3.26; proposition 3.3.30 and
	theorem 3.3.45 in \cite{MR2807904}
\end{proof}

\end{section}
\begin{section}{Birational and Weakly Birational Cohomology Theories}
		\label{Introd2}
		
	In this section, we construct the birational and weakly birational motivic stable homotopy categories.
	These are defined as left Bousfield localizations of $\TspectraX$ with respect to maps which are induced
	by open immersions with a numerical condition in the codimension of the closed complement (which is
	assumed to be smooth in the weakly birational case).  The existence of the left Bousfield localizations
	considered in this section follows immediately from theorems \ref{Tspectra-cellular} and \ref{Hirsch-Bousloc}.
	
\begin{lem}
		\label{lemma.relations-generators}
	Let $a,a',b,b',p,p'\geq 0$ be integers such that
	$a-p=a'-p'$ and $b-p=b'-p'$.  Assume that
	$p\geq p'$, then for every $Y\in Sm_{X}$, there is a
	weak equivalence in $\TspectraX$, which is
	natural with respect to $Y$
		\[	g^{a,b}_{p,p'}(Y):F_{p}(S^{a}\wedge 
				\Gm ^{b} \wedge
				Y_{+}) \rightarrow
			F_{p'}(S^{a'}\wedge \Gm ^{b'}\wedge
				Y_{+})
		\]
\end{lem}
\begin{proof}
	We have the following adjunction
	(see \cite[Def. 2.6.8]{MR2807904})
		\[	(F_{p},ev_{p},\varphi ):\mathcal M
			\rightarrow \TspectraX
		\]
	Using this adjunction, we define $g^{a,b}_{p,p'}(Y)$
	as adjoint to the identity map:
		\begin{align*}
			S^{a}\wedge \Gm ^{b}\wedge Y_{+}
				\xrightarrow{id}
			ev_{p}(F_{p'}(S^{a'}\wedge 
				\Gm ^{b'}\wedge Y_{+})) & \cong
				S^{p-p'}\wedge \Gm ^{p-p'}
			\wedge S^{a'}\wedge \Gm ^{b'}\wedge
			Y_{+} \\ & \cong
			S^{a}\wedge \Gm ^{b}\wedge Y_{+}
		\end{align*}
	Thus, it is clear that $g^{a,b}_{p,p'}(Y)$ is
	natural in $Y$, and it follows from
	\cite[Prop. 2.4.26]{MR2807904}
	that it is a weak equivalence in $\TspectraX$.
\end{proof}	
		
\begin{defi}[{see \cite[section 7.5]{MR0338129}}]
		\label{def.codimension}
	Let $Y\in Sch_{X}$, and $Z$ a closed subscheme of $Y$.  The \emph{codimension}
	of $Z$ in $Y$, $codim_{Y}Z$ is the infimum (over the generic points $z_{i}$ of $Z$)
	of the dimensions of the local rings $\mathcal O _{Y, z_{i}}$. 
\end{defi}

	Since $X$ is Noetherian of finite Krull dimension and $Y$ is of finite type over $X$, $codim_{Y}Z$ is always finite.
		
\begin{defi}
		\label{def.localizing-maps}
	We fix an arbitrary integer  $n\geq 0$, and consider the following set of 
	open immersions which have a closed complement of codimension at least $n+1$
\begin{align*}
	B_{n}=\{ \iota _{U,Y}:& U\rightarrow Y \text{ open immersion } |\\														
											 & Y\in Sm_{X}; Y \text{ irreducible};
									     (codim_{Y}Y\backslash U)\geq n+1
	\}
\end{align*}
	The letter $B$ stands for birational.	
\end{defi}

	Now we consider the left Bousfield localization of $\TspectraX$ with respect to a suitable set of
	maps induced by the families of open immersions $B_{n}$ described above.
	
\begin{defi}
		\label{def.localizationAmod-Birat}
	Let $n\in \mathbb Z$ be an arbitrary integer.
	\begin{enumerate}
		\item \label{def.localizationAmod-Birat.a}  Let $\nbiratTspectraX$ denote the
				left Bousfield localization of $\TspectraX$ 
				with respect to the set of maps 
					\[	sB_{n}=\{ F_{p}(\Gm ^{b}\wedge \iota _{U,Y}): b, p, r\geq 0, 
							b-p\geq n-r; \iota _{U,Y}\in B_{r} \}.
					\]
		\item	\label{def.localizationAmod-Birat.b}  Let $b^{(n)}$ denote its fibrant replacement functor 
					and $\stablenbiratTspectraX$ its associated homotopy category.
	\end{enumerate}
	For $n\neq 0$ we will call $\stablenbiratTspectraX$ the \emph{codimension $n+1$-birational motivic stable homotopy
	category,} and for $n=0$ we will call it the \emph{birational motivic stable homotopy category}.
\end{defi}

\begin{lem}
		\label{lemma.generators.becomeequiv.1}
	Let $n\in \mathbb Z$ be an arbitrary integer.  Then for every $a\geq 0$, the maps
		\[	S^{a}\wedge sB_{n}=\{ F_{p}(S^{a}\wedge \Gm ^{b}\wedge 
							\iota _{U,Y}): 
							b, p, r\geq 0, b-p \geq n-r; \iota _{U,Y}\in B_{r} \}
		\]
	are weak equivalences in $\nbiratTspectraX$.
\end{lem}
\begin{proof}
	Let $F_{p}(\Gm ^{b}\wedge \iota _{U,Y})\in sB_{n}$ with $\iota _{U,Y}
	\in B_{r}$.  Both $F_{p}(\Gm ^{b}\wedge U_{+})$ and $F_{p}(\Gm ^{b}
	\wedge Y_{+})$ are cofibrant in $\TspectraX$ (see
	\cite[Props. 2.4.17, 2.6.18 and Thm. 2.6.30]{MR2807904})
	and hence also in $\nbiratTspectraX$.   By construction,
	$F_{p}(\Gm ^{b}\wedge \iota _{U,Y})$ is a weak equivalence  in 
	$\nbiratTspectraX$; and \cite[Thm. 4.1.1.(4)]{MR1944041} implies
	that $\nbiratTspectraX$ is a simplicial model category.
	Thus, it follows from
	Ken Brown's lemma (see \cite[lemma 1.1.12]{MR1650134})
	that $F_{p}(S^{a}\wedge \Gm ^{b}\wedge \iota _{U,Y})$ is also a weak
	 equivalence in $\nbiratTspectraX$ for every $a\geq 0$.
\end{proof}

\begin{prop}
		\label{prop.charac.fibrants1}
	Let $E$ be an arbitrary symmetric $T$-spectrum.  Then
	$E$ is fibrant in $\nbiratTspectraX$ if and only if the following conditions hold:
		\begin{enumerate}
			\item	\label{prop.charac.fibrants1.a}  $E$ is fibrant in $\TspectraX$.
			\item	\label{prop.charac.fibrants1.b}		For every $a, b, p, r\geq 0$
							 such that $b-p\geq n-r$; and every $\iota _{U,Y} \in B_{r}$,
							 the induced map
								\[ \xymatrix{\Hom _{\stablehomotopyX}(F_{p}(S^{a}\wedge
										 \Gm ^{b} \wedge Y_{+}), E) \ar[r]_-{\cong}^-{\iota _{U,Y}
										 ^{\ast}} &
										\Hom _{\stablehomotopyX}(F_{p}(S^{a}\wedge \Gm ^{b}
										\wedge U_{+}), E)}
								\]
							is an isomorphism.
		\end{enumerate}
\end{prop}
\begin{proof}
	($\Rightarrow$): Since the identity functor
		\[	id: \TspectraX \rightarrow \nbiratTspectraX
		\]
	is a left Quillen functor, the conclusion follows from the derived
	adjunction 
		\[		(Q, b^{(n)},\varphi ):\stablehomotopyX \rightarrow
					 \stablenbiratTspectraX
		\]
	together with lemma \ref{lemma.generators.becomeequiv.1}.
	
	($\Leftarrow$):	Assume that $E$ satisfies (\ref{prop.charac.fibrants1.a})
		and (\ref{prop.charac.fibrants1.b}).  Let $\omega _{0}$,
		$\eta _{0}$ denote the base points of the pointed
		simplicial sets $\mathrm{Map}_{\ast}(F_{p}(\Gm
		^{b}\wedge Y_{+}),E)$ and $\mathrm{Map}_{\ast}(F_{p}(\Gm ^{b}
		\wedge U_{+}),E)$ respectively.  Since
		$F_{p}(\Gm ^{b}\wedge Y_{+})$ and
		$F_{p}(\Gm ^{b}\wedge U_{+})$ are always cofibrant,
		by \cite[Def. 3.1.4(1)(a) and Thm. 4.1.1(2)]{MR1944041}
		it is enough to show that every map in $sB_{n}$ induces
		a weak equivalence of simplicial sets:
			\[	\xymatrix{ \mathrm{Map} _{\ast}(F_{p}(\Gm ^{b}
			\wedge Y_{+}),E) \ar[r]^-{\iota _{U,Y}^{\ast}}& 
				\mathrm{Map} _{\ast}(F_{p}(\Gm ^{b}
				\wedge U_{+}),E)}
			\]
		
		Since $\TspectraX$ is a pointed simplicial model category,
		we observe that lemma 6.1.2 in \cite{MR1650134} and remark
		2.4.3(2) in \cite{MR2807904} imply that the following diagram
		is commutative for $a\geq 0$ and all the vertical arrows
		are isomorphisms
			\[	\xymatrix{\pi _{a,\omega _{0}}\mathrm{Map} _{\ast}(F_{p}(\Gm ^{b} 
			\wedge Y_{+}),E) \ar[dr]_-{\iota _{U,Y}^{\ast}} \ar[dd]_-{\cong}& \\
			& \pi _{a,\eta _{0}}\mathrm{Map} _{\ast}(F_{p}(\Gm ^{b} 
			\wedge U_{+}),E) \ar[dd]^-{\cong} \\
			\Hom _{\stablehomotopyX}(F_{p}(S^{a}\wedge
			\Gm ^{b} \wedge Y_{+}),E) \ar[dr]^-{\iota _{U,Y}^{\ast}} & \\
			& \Hom _{\stablehomotopyX}(F_{p}(S^{a}\wedge
			\Gm ^{b} \wedge U_{+}),E)}
			\]
		by hypothesis, the bottom row is an isomorphism, hence
		the top row is also an isomorphism.  This implies that 
		for every map in $sB_{n}$, the induced map
			\[	\xymatrix{ \mathrm{Map} _{\ast}(F_{p}(\Gm ^{b}
			\wedge Y_{+}),E) \ar[r]^-{\iota _{U,Y}^{\ast}}& 
				\mathrm{Map} _{\ast}(F_{p}(\Gm ^{b}
				\wedge U_{+}),E)}
			\]
		is a weak equivalence when it is restricted to the path component
		of $\mathrm{Map}_{\ast}(F_{p}(\Gm ^{b}\wedge Y_{+}),E)$
		containing $\omega _{0}$.  This holds in particular for
			\[		\xymatrix{ \mathrm{Map} _{\ast}(F_{p+1}(\Gm ^{b+1}
			\wedge Y_{+}),E) \ar[r]^-{\iota _{U,Y}^{\ast}}& 
				\mathrm{Map} _{\ast}(F_{p+1}(\Gm ^{b+1}
				\wedge U_{+}),E)}
			\]
		Therefore, the following map is a weak equivalence of
		pointed simplicial sets, since taking $S^{1}$-loops kills the path
		components that do not contain the base point
			\[		\xymatrix{ \mathrm{Map}_{\ast}(S^{1},
						\mathrm{Map} _{\ast}(F_{p+1}(\Gm ^{b+1}
						\wedge Y_{+}),E))\ar[d] \\
						 \mathrm{Map}_{\ast}(S^{1}, 
						\mathrm{Map} _{\ast}(F_{p+1}(\Gm ^{b+1}
						\wedge U_{+}),E))}
			\]
		
		Now, since $\TspectraX$ is a simplicial model category
		we deduce that the rows in the following commutative diagram
		are isomorphisms
			\[	\xymatrix@C=1pc{\mathrm{Map}_{\ast}(S^{1},
						\mathrm{Map} _{\ast}(F_{p+1}(\Gm ^{b+1}
						\wedge Y_{+}),E)) \ar[dr]_-{\cong} 
						\ar[dd]_-{\iota _{U,Y}^{\ast}}& \\
			& \mathrm{Map} _{\ast}(F_{p+1}(S^{1}\wedge \Gm ^{b+1}
			\wedge Y_{+}),E) \ar[dd]^-{\iota _{U,Y}^{\ast}} \\
			\mathrm{Map}_{\ast}(S^{1},
						\mathrm{Map} _{\ast}(F_{p+1}(\Gm ^{b+1}
						\wedge U_{+}),E)) \ar[dr]_-{\cong} & \\
			& \mathrm{Map} _{\ast}(F_{p+1}(S^{1}\wedge \Gm ^{b+1}
			\wedge U_{+}),E)}
			\]
		Thus, by the three out of two property for weak equivalences, we
		conclude that
			\[		\xymatrix{ \mathrm{Map} _{\ast}(F_{p+1}(S^{1}\wedge 
						\Gm ^{b+1}
						\wedge Y_{+}),E) \ar[r]^-{\iota _{U,Y}^{\ast}}& 
						\mathrm{Map} _{\ast}(F_{p+1}(S^{1}\wedge \Gm ^{b+1}
						\wedge U_{+}),E)}
			\]
		is also a weak equivalence of pointed simplicial sets.  Finally, lemma \ref{lemma.relations-generators}
implies that the following diagram is commutative and
the vertical arrows are weak equivalences in 
$\TspectraX$
			\[  \xymatrix{ \mathrm{Map} _{\ast}
			        (F_{p+1}(S^{1}\wedge \Gm ^{b+1}
				\wedge Y_{+}),E) \ar[r]^-{\iota _{U,Y}
				^{\ast}} &
				 \mathrm{Map} _{\ast}(
				F_{p+1}(S^{1}\wedge \Gm ^{b+1}
				\wedge U_{+}),E) \\
				\mathrm{Map} _{\ast}(F_{p}(\Gm 
				^{b}\wedge Y_{+}),E) \ar[r]^-{\iota 
				_{U,Y}^{\ast}} \ar[u]^-{g^{1,b+1}_{p+1,p}(Y)^{\ast}} & 
				\mathrm{Map} 
				_{\ast}(F_{p}(\Gm ^{b}\wedge 
				U_{+}),E) \ar[u]_-{g^{1,b+1}_{p+1,p}(U)^{\ast}}}
			\]
	Thus, we conclude by the  two out of three property
	for weak equivalences that the bottom arrow
	is also a weak equivalence in $\TspectraX$.
\end{proof}

\begin{prop}
		\label{prop.properties-homotbirat}
	The homotopy category $\stablenbiratTspectraX$ is a compactly generated triangulated category
	 in the sense of Neeman \cite[Def. 1.7]{MR1308405}.
\end{prop}
\begin{proof}
	We will prove first that $\stablenbiratTspectraX$ is a triangulated category.  For this, it is enough
	to show that the smash product with the simplicial circle induces a Quillen equivalence
	(see \cite[sections I.2, I.3]{MR0223432})
		\[ (-\wedge S^{1},\Omega _{S^{1}}-,\varphi ): \nbiratTspectraX \rightarrow \nbiratTspectraX
		\]
	It follows from \cite[Thm. 4.1.1.(4)]{MR1944041} that this adjunction is a Quillen
	adjunction, and the same argument as in \cite[Cor. 3.2.38]{MR2807904} (replacing 
	\cite[Prop. 3.2.32]{MR2807904} with proposition \ref{prop.charac.fibrants1})
	allows us to conclude that it is a Quillen equivalence. 
	
	Finally, since $\stablehomotopyX$ is a compactly generated triangulated category 
	(see \cite[Prop. 3.1.5]{MR2807904}) and the identity functor is a left Quillen functor
		\[	id:\TspectraX \rightarrow \nbiratTspectraX
		\]
	it follows from the derived adjunction 
		\[		(Q, b^{(n)},\varphi ):\stablehomotopyX \rightarrow 
					\stablenbiratTspectraX
		\]
	that $\stablenbiratTspectraX$ is also compactly generated, having
	exactly the same set of generators as $\stablehomotopyX$.
\end{proof}

\begin{defi}
		\label{def.localizing-maps2}
	We fix an arbitrary integer  $n\geq 0$, and consider the following set of 
	open immersions with smooth closed complement	of codimension at least $n+1$
\begin{align*}
	WB_{n}=\{ \iota _{U,Y}: &U\rightarrow Y \text{ open immersion } |\\
									& Y, Z=Y\backslash U \in Sm_{X}; Y \text{ irreducible};
									      (codim_{Y}Z)\geq n+1
	\}
\end{align*}	
\end{defi}

	Notice that every map in $WB_{n}$ is  also in $B_{n}$, but the converse doesn't hold.  The
	reason to consider maps $\iota _{U,Y}$ in $WB_{n}$ is that if the closed complement is smooth, then
	the Morel-Voevodsky homotopy purity theorem (see \cite[Thm. 2.23]{MR1813224})
	characterizes the homotopy cofibre
	of $\iota _{U,Y}$ in terms of the Thom space of the normal bundle for the closed immersion
	$Y\backslash U \rightarrow Y$.

\begin{defi}
		\label{def.localizationAmod-weaklyBirat}
	Let $n\in \mathbb Z$ be an arbitrary integer.
	\begin{enumerate}
		\item	\label{def.localizationAmod-weaklyBirat.a}  Let $\nwbiratTspectraX$ denote
					the left Bousfield localization of $\TspectraX$ with respect to the set of maps 
						\[	sWB_{n}=\{ F_{p}(\Gm ^{b}\wedge \iota _{U,Y}): b, p, r\geq 0, 
								b-p\geq n-r; \iota _{U,Y}\in WB_{r} \}.
						\]
		\item	\label{def.localizationAmod-weaklyBirat.b}  Let $wb^{(n)}$ denote its fibrant replacement 
					functor and $\stablenwbiratTspectraX$ its associated homotopy category.
	\end{enumerate}
	For $n\neq 0$ we will call $\stablenwbiratTspectraX$ the 
	\emph{codimension $n+1$-weakly birational motivic stable homotopy category},
	and for $n=0$ we will call it the \emph{weakly birational motivic stable homotopy category}.
\end{defi}

\begin{prop}
		\label{prop.charac.fibrants2}
	Let $E$ be an arbitrary symmetric $T$-spectrum.  Then
	$E$ is fibrant in $\nwbiratTspectraX$ if and only if the following 
	conditions hold:
		\begin{enumerate}
			\item	\label{prop.charac.fibrants2.a}  $E$ is fibrant in $\TspectraX$.
			\item	\label{prop.charac.fibrants2.b}		For every $a, b, p, r\geq 0$
							 such that $b-p\geq n-r$; and every $\iota _{U,Y} \in WB_{r}$,
							 the induced map
								\[ \xymatrix{\Hom _{\stablehomotopyX}(F_{p}(S^{a}\wedge
										 \Gm ^{b} \wedge Y_{+}), E) \ar[r]_-{\cong}^-{\iota _{U,Y}
										 ^{\ast}} &
										\Hom _{\stablehomotopyX}(F_{p}(S^{a}\wedge \Gm ^{b}
										\wedge U_{+}), E)}
								\]
							is an isomorphism.
		\end{enumerate}
\end{prop}
\begin{proof}
	The proof is exactly the same as in proposition 
	\ref{prop.charac.fibrants1}.
\end{proof}

\begin{prop}
		\label{prop.properties-homotweakbirat}
	The homotopy category $\stablenwbiratTspectraX$ is a compactly generated triangulated category in the sense of Neeman.
\end{prop}
\begin{proof}
	The proof is exactly the same as in proposition \ref{prop.properties-homotbirat}.
\end{proof}

\begin{prop}
		\label{prop.genericsmoothness-Quillenequiv}
	Assume that the base scheme $X=\mathrm{Spec}\; k$, with $k$ a perfect field, 
	then the Quillen adjunction:
		$$(id,id,\varphi):\nwbiratTspectraX \rightarrow \nbiratTspectraX
		$$
	is a Quillen equivalence.
\end{prop}
\begin{proof}
	Consider the following commutative diagram
		\[	\xymatrix{ & \TspectraX \ar[dl]_-{id} \ar[dr]^-{id}& \\
							\nwbiratTspectraX \ar@{-->}[rr]_-{id}& & \nbiratTspectraX}
		\]
	where the solid arrows are left Quillen functors.  Clearly, 
	$WB_{r}\subseteq B_{r}$ for every $r\geq 0$, so
	$sWB_{n}\subseteq sB_{n}$, and we conclude that every $sWB_{n}$-local equivalence is
	a $sB_{n}$-local equivalence.  Therefore,  the universal
	property of left Bousfield localizations implies that the horizontal 
	arrow is also a left Quillen functor.
	
	The universal property for left Bousfield localizations also implies that
	it is enough to show that all the maps in 
		\[	sB_{n}=\{ F_{p}(\Gm ^{b}\wedge \iota _{U,Y}): b, p, r\geq 0, 
					b-p\geq n-r; \iota _{U,Y}\in B_{r} \}
		\]
	become weak equivalences in $\nwbiratTspectraX$.  Given
	$F_{p}(\Gm ^{b}\wedge \iota _{U,Y})\in sB_{n}$
	with $\iota_{U,Y}\in B_{r}$, we proceed by induction
	on the dimension of $Z=Y\backslash U$.  If $\kdim \; Z=0$, then
	$Z\in Sm_{X} $ since $k$ is a perfect field (and we are considering
	$Z$ with the reduced scheme structure), hence
	$F_{p}(\Gm ^{b}\wedge \iota _{U,Y})\in sWB_{n}$ and then a
	weak equivalence in $\nwbiratTspectraX$.
	
	If $\kdim \; Z>0$, then
	we consider the singular locus $Z_{s}$ of $Z$ over $X$.  We have that
	$\kdim \; Z_{s}<\kdim \; Z$ since $k$ is a perfect field.  Therefore, by
	induction on the dimension $F_{p}(\Gm ^{b}\wedge \iota _{V,Y})$
	is a weak equivalence in $\nwbiratTspectraX$, where $V=Y\backslash
	Z_{s}$.  On the other hand, $F_{p}(\Gm ^{b}\wedge \iota _{U,V})$
	is also a weak equivalence in $\nwbiratTspectraX$ since 
	$\iota _{U,V}$ is also in $B_{r}$ and its closed
	complement $V\backslash U=Z\backslash Z_{s}$ is smooth
	over $X$, by construction of $Z_{s}$.
	
	But $F_{p}(\Gm ^{b}\wedge \iota _{U,Y})=F_{p}(\Gm ^{b}\wedge 
	\iota _{V,Y})\circ F_{p}(\Gm ^{b}\wedge \iota _{U,V})$, so
	by the two out of three property for weak equivalences we
	conclude that $F_{p}(\Gm ^{b}\wedge \iota _{U,Y})$ is a weak equivalence
	in $\nwbiratTspectraX$.
\end{proof}

\end{section}
\begin{section}{A Characterization of the Slices}
		\label{sect-1}
				
	This section contains our main results.  We give a characterization of the slices in terms
	of effectivity and birational conditions (in the sense of definition \ref{def.n-birational}), 
	and we also show that there is an
	equivalence between the notion of orthogonality (see definition \ref{def.orthogonality})
	and weak birationality (see definition \ref{def.n-birational}).

\begin{defi}
		\label{def.n-birational}
	Let $E\in \TspectraX$ be a symmetric $T$-spectrum and $n\in \mathbb Z$.
	\begin{enumerate}
		\item \label{def.n-birational.a}  We will say that $E$ is \emph{$n+1$-birational} 
					(respectively \emph{weakly $n+1$-birational}),
					if $E$ is fibrant in $\nbiratTspectraX$ (respectively $\nwbiratTspectraX$).  If $n=0$,
					we will simply say that $E$ is \emph{birational} (respectively \emph{weakly birational}).
		\item	\label{def.n-birational.b}  We will say that $E$ is an
	\emph{$n$-slice} if $E$ is isomorphic in $\stablehomotopyX$ to $s_{n}(E')$ for some symmetric 
	$T$-spectrum
	$E'$.
	\end{enumerate}
\end{defi}

\begin{defi}
		\label{def.cofibres}
	\begin{enumerate}
		\item \label{def.cofibres.a}  Let $\iota _{U,Y}$ be an open immersion in $Sm_{X}$.  Let
				$Y/U$ denote the pushout of the following
				diagram in $\mathcal M$ (i.e. the homotopy cofibre of $\iota _{U,Y}$ in $\mathcal M$)
					\[	\xymatrix{ U_{+} \ar[r]^-{\iota _{U,Y}} \ar[d]& Y_{+} \ar[d]\\
							X \ar[r] & Y/U}
					\]
		\item	\label{def.cofibres.b}  Given a vector bundle $\pi: V\rightarrow Y$ with $Y\in Sm_{X}$, let $Th(V)$ denote the Thom
	space of $V$, i.e. $V/(V\backslash \sigma _{0}(Y))$, where 
	$\sigma _{0}:Y\rightarrow V$ denotes the zero section of $V$.
	\end{enumerate}
\end{defi}

\begin{lem}
		\label{lemma.characterizing-cofibres}
	Let $\iota _{U,Y}\in WB_{r}$ for some $r\geq 0$, and
	let $a,b,p\geq 0$ be arbitrary integers such that $b-p\geq n-r$.  Then
		\[		F_{p}(S^{a}\wedge \Gm ^{b}\wedge Y/U)\in 
					\Sigma _{T}^{n+1}\stablehomotopyX^{\mathit{eff}}
		\]
\end{lem}
\begin{proof}
	Since $\Sigma _{T}^{n+1}\stablehomotopyX^{\mathit{eff}}$ is a triangulated category, it is 
	enough to consider the case $a=0$.  It is also clear that it suffices to show that $F_{0}(Y/U)\in 
	\Sigma _{T}^{r+1}\stablehomotopyX^{\mathit{eff}}$.
	
	Now, it follows from the Morel-Voevodsky homotopy purity theorem 
	(see \cite[Thm. 2.23]{MR1813224})
	that there is an isomorphism in $\stablehomotopyX$
		\[ F_{0}(Y/U) \rightarrow F_{0}(Th(N))
		\]
	where Th(N) is the Thom space of the normal bundle $N$ of the (smooth) complement
	$Z$ of U in $Y$:
		\[	e:Y\backslash U=Z \rightarrow Y
		\]
	But, $\iota _{U,Y}\in WB_{r}$;
	so $e$ is a regular embedding of codimension $c$ at least $r+1$, hence $N$ is a vector bundle
	of rank at least $r+1$.  Therefore, if $N$ is a trivial vector bundle we conclude from
	\cite[Prop. 2.17(2)]{MR1813224} that 
		\[	F_{0}(Th(N))\cong F_{0}(S^{c}\wedge \Gm ^{c}\wedge Z_{+}) \in 
			\Sigma _{T}^{c}\stablehomotopyX^{\mathit{eff}}
			\subseteq \Sigma _{T}^{r+1}\stablehomotopyX^{\mathit{eff}}
		\]
	Finally, we conclude in the general case by choosing a Zariski cover of $Z$ which trivializes
	$N$ and using the Mayer-Vietoris property for Zariski covers.
\end{proof}

\begin{lem}
		\label{lemma.smashGm=cofibre}
	Let $U\in Sm_{X}$.  Consider the 
	open immersion in $Sm_{X}$
		\[	m _{U}:\mathbb A ^{1}_{U}\backslash U \rightarrow
				 \mathbb A ^{1}_{U}
		\]
	given by the complement of the zero section.  Then
	$m _{U}\in WB_{0}$, and there exists a weak equivalence in $\TspectraX$
	between its homotopy cofibre in $\mathcal M$, $\mathbb A ^{1}_{U}/(\mathbb 
	A ^{1}_{U}\backslash U)$ 
	and $S^{1}\wedge \Gm \wedge U_{+}$ 
		\[	t_{U}:\mathbb A ^{1}_{U}/(\mathbb A ^{1}_{U}\backslash U)
				\rightarrow S^{1}\wedge \Gm \wedge U_{+}
		\]
\end{lem}
\begin{proof}
	Since the zero section $i_{0}:U\rightarrow \mathbb A ^{1}_{U}$ is a closed embedding of
	codimension $1$ between smooth schemes over $X$, it follows from the definition of 
	$WB_{0}$ that
	$m_{U}\in WB_{0}$.  Finally, \cite[Prop. 2.17(2)]{MR1813224} implies the existence of the
	weak equivalence $t_{U}$.
\end{proof}

\begin{prop}
		\label{prop.characterization-orthogonality}
	Let $E\in \TspectraX$ be a symmetric $T$-spectrum and $n\in \mathbb Z$.  Consider
	the following conditions:
		\begin{enumerate}
			\item \label{prop.characterization-orthogonality.a}  $E$
				is fibrant in $\nplusoneorthogonalTspectraX$. 
			\item \label{prop.characterization-orthogonality.b}  $E$ is weakly $n+1$-birational
				(see definition \ref{def.n-birational}(\ref{def.n-birational.a})).
			\item \label{prop.characterization-orthogonality.c} $E$ is $n+1$-birational
				(see definition \ref{def.n-birational}(\ref{def.n-birational.a})).
		\end{enumerate}
	Then (\ref{prop.characterization-orthogonality.a}) and (\ref{prop.characterization-orthogonality.b})
	are equivalent.  In addition, if the base scheme $X=\mathrm{Spec}\; k$, with $k$ a perfect field, then
	(\ref{prop.characterization-orthogonality.a}), (\ref{prop.characterization-orthogonality.b}) and
	(\ref{prop.characterization-orthogonality.c}) are equivalent.
\end{prop}
\begin{proof}
	(\ref{prop.characterization-orthogonality.a})$\Rightarrow$(\ref{prop.characterization-orthogonality.b}):
	Assume that $E$ is fibrant in $\nplusoneorthogonalTspectraX$.
	By proposition \ref{prop.charac.fibrants2} it suffices to show that
	for every $a,b,p,r\geq 0$ with $b-p\geq n-r$, and every
	$\iota _{U,Y}\in WB_{r}$; the induced map
		\[ \xymatrix{\Hom _{\stablehomotopyX}(F_{p}(S^{a}\wedge
					 \Gm ^{b} \wedge Y_{+}), E) \ar[r]_-{\cong}^-{\iota 
					 _{U,Y}^{\ast}} & \Hom _{\stablehomotopyX}(F_{p}(S
					 ^{a}\wedge \Gm ^{b}\wedge U_{+}), E)}
		\]
	is an isomorphism.  We observe that
		\[		F_{p}(S^{a}\wedge \Gm ^{b}\wedge -): \mathcal M
					\rightarrow \TspectraX
		\]
	is a left Quillen functor, therefore the following
		\[	\xymatrix@C=0.9pc{F_{p}(S^{a}\wedge \Gm ^{b}\wedge U_{+}) 
						\ar[rrrr]^-{F_{p}(S^{a}\wedge \Gm ^{b} \wedge
						\iota _{U,Y})} &&&&
						F_{p}(S^{a}\wedge \Gm ^{b}\wedge Y_{+}) \ar[r]&
						F_{p}(S^{a}\wedge \Gm ^{b}\wedge Y/U)}
		\]
	is a cofibre sequence in $\TspectraX$.  However, 
	$\stablehomotopyX$ is a triangulated category and lemma \ref{lemma.relations-generators} 
	implies that 
		\[F_{p+1}(S^{a}\wedge \Gm ^{b+1}\wedge Y/U)\cong \Omega _{S^{1}}\circ R\circ F_{p}(S^{a}\wedge 
			\Gm ^{b}\wedge Y/U)
		\]
	are isomorphic in $\stablehomotopyX$, where $R$ denotes a fibrant replacement functor in $\TspectraX$.
	Hence it suffices to show that
		\[		\Hom _{\stablehomotopyX}(F_{p+1}(S^{a}\wedge
				\Gm ^{b+1}\wedge Y/U),E)= \Hom _{\stablehomotopyX}(F_{p}(S^{a}\wedge
				\Gm ^{b}\wedge Y/U),E)=0
		\]
	But this follows from lemma \ref{lemma.characterizing-cofibres}
	together with \cite[Prop. 3.3.30]{MR2807904}, since we are
	assuming that $E$ is fibrant in $\nplusoneorthogonalTspectraX$.
	
	(\ref{prop.characterization-orthogonality.b})$\Rightarrow$(\ref{prop.characterization-orthogonality.a})
	Assume that $E$ is $n+1$-weakly birational.  Then,
	proposition 3.3.30 in \cite{MR2807904} implies that it suffices
	to show that
		\[		\Hom _{\stablehomotopyX}(F_{p}(S^{a}\wedge
					\Gm ^{b}\wedge U_{+}),E)=0
		\]
	for every $F_{p}(S^{a}\wedge \Gm ^{b}\wedge U_{+})\in 
	C^{n+1}_{\mathit{eff}}$.  
	
	The same argument as in lemma
	\ref{lemma.generators.becomeequiv.1}
	implies that it is enough to consider the case when
	$F_{p}(\Gm ^{b}\wedge U _{+})\in C_{\mathit{eff}}^{n+1}$.
	Moreover, we can further reduce to the case where $b,p\geq 1$
	and $F_{p}(S^{1}\wedge \Gm ^{b}\wedge U _{+})
	\in C_{\mathit{eff}}^{n+1}$.
	In effect, if $F_{p}(\Gm ^{b}\wedge U _{+})\in C_{\mathit{eff}}^{n+1}$,
	then lemma \ref{lemma.relations-generators} implies that the natural map 
		\[	g^{1,b+1}_{p+1,p}(U):F_{p+1}(S^{1}\wedge \Gm^{b+1} \wedge U_{+})\rightarrow 
			F_{p}(\Gm ^{b}\wedge U _{+})
		\] 
	is a weak equivalence in $\TspectraX$.
	
	Now, it follows from lemma \ref{lemma.smashGm=cofibre}, that
	if $b\geq 1$, and $0-p+(b-1)\geq n$ (i.e. $b-p\geq n+1$); then
	$F_{p}(\Gm ^{b-1}\wedge m_{U})\in sWB_{n}$, i.e. a weak equivalence
	in $\nwbiratTspectraX$.
	
	Since $\stablenwbiratTspectraX$ is a triangulated category,
	$id:\TspectraX \rightarrow \nwbiratTspectraX$ is a left Quillen functor, and
	$F_{p}(\Gm ^{b-1}\wedge (\mathbb A ^{1}_{U}/(\mathbb A ^{1}_{U}\backslash U_{+})))$ 
	is the homotopy
	cofibre of $F_{p}(\Gm ^{b-1}\wedge m_{U})$; we deduce that $E$ being
	$n+1$-weakly birational implies that
		\[  \Hom _{\stablehomotopyX}(F_{p}(\Gm ^{b-1}\wedge 
			(\mathbb A ^{1}_{U}/(\mathbb A ^{1}_{U}\backslash U_{+}))),E)=0
		\]
	Finally, it follows from lemma \ref{lemma.smashGm=cofibre} that the following groups are isomorphic
		\begin{align*}
			0 & = \Hom _{\stablehomotopyX}(F_{p}(\Gm ^{b-1}\wedge 
			(\mathbb A ^{1}_{U}/(\mathbb A ^{1}_{U}\backslash U_{+}))),E) \\ & \cong 
			\Hom _{\stablehomotopyX}(F_{p}(S^{1}\wedge \Gm ^{b}\wedge U_{+}),E) 
		\end{align*}
	
	(\ref{prop.characterization-orthogonality.b})$\Leftrightarrow$(\ref{prop.characterization-orthogonality.c}):  This
	follows directly from proposition \ref{prop.genericsmoothness-Quillenequiv}.
\end{proof}

\begin{thm}
		\label{thm.orthogonal=wbirationalTspectra}
	The Quillen adjunction
		\[ (id, id, \varphi ): \nwbiratTspectraX \rightarrow \nplusoneorthogonalTspectraX
		\]
	is a Quillen equivalence.  In addition, if the base scheme $X=\mathrm{Spec}\; k$, with $k$ a perfect field,
	then the Quillen adjunction
		\[ (id, id, \varphi ): \nbiratTspectraX \rightarrow \nplusoneorthogonalTspectraX
		\]
	is also a Quillen equivalence.
\end{thm}
\begin{proof}
	We show first that $\nwbiratTspectraX$ and $\nplusoneorthogonalTspectraX$ are Quillen equivalent.
	Since $\nwbiratTspectraX$, $\nplusoneorthogonalTspectraX$ are both left Bousfield localizations
	of $\TspectraX$, we deduce that they are simplicial model categories with the same cofibrant 
	replacement functor $Q$.  Thus, it suffices to show that they have the same class of
	weak equivalences.
	
	However,
	proposition \ref{prop.characterization-orthogonality} implies that $\nwbiratTspectraX$, and
	$\nplusoneorthogonalTspectraX$ also have the same
	class of fibrant objects. Therefore, it follows from \cite[Thm. 9.7.4]{MR1944041} that they 
	have exactly the same
	class of weak equivalences.
	
	Finally, if the base scheme is a perfect field, by proposition \ref{prop.genericsmoothness-Quillenequiv}
	we conclude that $\nwbiratTspectraX$ and $\nbiratTspectraX$ are Quillen equivalent.
\end{proof}
		
\begin{thm}
		\label{1st-characterization-slices}
	Let $E$ be fibrant in $\TspectraX$.
	Then $E$ is an $n$-slice (see definition \ref{def.n-birational}(\ref{def.n-birational.b})) 
	if and only if the following conditions hold:
		\begin{description}
			\item[S1]		$E$ is $n$-effective, i.e.
				 $E\in \neffstablehomotopyX$.						
			\item[S2]		$E$ is $n+1$-weakly
				 birational.
		\end{description}
	In addition, if the base scheme $X=\mathrm{Spec}\; k$, with $k$ a perfect field, then
	$E$ is an $n$-slice if and only if the following conditions hold:
		\begin{description}
			\item[GSS1]	$E$ is $n$-effective, i.e. $E\in \neffstablehomotopyX$.
			\item[GSS2]	$E$ is $n+1$-birational.
		\end{description}
\end{thm}
\begin{proof}
	Assume that $E$ is an $n$-slice.  Then 
	theorems	 \ref{thm.modelstructures-slicefiltration}(\ref{thm.modelstructures-slicefiltration.a}) 
	and \ref{thm.modelstructures-slicefiltration}(\ref{thm.modelstructures-slicefiltration.c}) imply
	that $E$ is $n$-effective and
	fibrant in $\nplusoneorthogonalTspectraX$.  Hence,
	proposition \ref{prop.characterization-orthogonality} implies
	that $E$ is also $n+1$-weakly birational.
	
	Now we assume that $E$ satisfies the conditions
	\textbf{S1} and \textbf{S2} above.  Then, proposition
	\ref{prop.characterization-orthogonality} implies that
	$E$ is fibrant in $\nplusoneorthogonalTspectraX$.
	Therefore, theorem \ref{thm.modelstructures-slicefiltration}(\ref{thm.modelstructures-slicefiltration.c})
	implies that $E$ is isomorphic in $\stablehomotopyX$
	to its own $n$-slice $s_{n}(E)$.
	
	Finally, if the base scheme is a perfect field, then by
	proposition \ref{prop.characterization-orthogonality} the conditions \textbf{S2}
	and \textbf{GSS2} are equivalent; hence we can conclude applying the same argument
	as above.
\end{proof}

\begin{rmk}
		\label{rmk.unstable.case}
	Notice that theorem \ref{thm.orthogonal=wbirationalTspectra}
	implies that it is possible to
	construct the slice filtration directly from the Quillen model categories $\nwbiratTspectraX$ described in
	definition \ref{def.localizationAmod-weaklyBirat} without
	making any reference to the effective categories $\Sigma _{T}^{q}\stablehomotopyX^{\mathit{eff}}$.  
	One of the interesting consequences of this fact is that it is possible to obtain 
	analogues of the slice filtration in the unstable setting, since the suspension with respect 
	to $T$ or $S^{1}$ does not play an essential role in
	the construction of $\nwbiratTspectraX$, i.e. we could consider the left Bousfield localization
	of the motivic unstable homotopy category $\mathcal M$ with respect to
	the maps in definition \ref{def.localizing-maps2}.  
	We will study the details of this construction in a future work.
\end{rmk}
				
\end{section}
\begin{section}{Some Computations}
		\label{sect-2}

	In this section we use the characterization of the slices obtained in theorem
	\ref{1st-characterization-slices} to describe the slices of projective spaces, Thom spaces and  blow ups.
	
	To simplify the notation, given a simplicial presheaf $K\in \mathcal M$ or a map $f\in \mathcal M$; let
	$s_{j}(K)$, $s_{j}(f)$ (respectively $s_{<j}(K)$, $s_{<j}(f)$) denote $s_{j}(F_{0}(K))$, $s_{j}(F_{0}(f))$
	(respectively $s_{<j}(F_{0}(K))$, $s_{<j}(F_{0}(f))$).
	
\begin{lem}
		\label{lemma.computingslices-via-birinvs}
	Let $g:E\rightarrow F$ be a map in $\stablehomotopyX$ such that
	$s_{<n}(g)$ and $s_{<n+1}(g)$ are both isomorphisms in $\stablehomotopyX$.
	Then the $n$-slice of $g$, $s_{n}(g)$ is also an isomorphism in $\stablehomotopyX$.
\end{lem}
\begin{proof}
	It follows from \cite[Prop. 3.1.19]{MR2807904} that the rows in the following commutative diagram
	are distinguished triangles in $\stablehomotopyX$
		\[	\xymatrix{s_{n}(E) \ar[d]_-{s_{n}(g)} \ar[r] & s_{<n+1}(E)\ar[d]^-{s_{<n+1}(g)} \ar[r] & s_{<n}(E) 
							\ar[d]^-{s_{<n}(g)} \ar[r] & S^{1}\wedge s_{n}(E)\ar[d]^-{S^{1}\wedge s_{n}(g)} \\
							s_{n}(F) \ar[r]  & s_{<n+1}(F) \ar[r] & s_{<n}(F) \ar[r] & S^{1}\wedge s_{n}(F)}
		\]
	Thus, we conclude that $s_{n}(g)$ is also an isomorphism in $\stablehomotopyX$.
\end{proof}

	Consider $Y\in Sm_{X}$.  Let  $\mathbb P ^{n}(Y)$ denote the trivial projective bundle of rank $n$ over $Y$, and
	let $\mathbb P ^{\infty}(Y)$ denote the colimit in $\mathcal M$ of the
	following filtered diagram
		\begin{equation}
				\label{eq.inf.proj.space}
			\mathbb P^{0}(Y)\rightarrow \mathbb P ^{1}(Y)\rightarrow \cdots \rightarrow 
			\mathbb P^{n}(Y)\rightarrow \cdots
		\end{equation}
	given by the inclusions of the respective hyperplanes at infinity.  
	
\begin{thm}
		\label{thm.slices.proj.space}
	Let $Y\in Sm_{X}$.  Then
	for any integer $j\leq n$, the diagram \ref{eq.inf.proj.space} induces the following
	isomorphisms in $\stablehomotopyX$
					\[	\xymatrix@C=1.5pc@R=0.5pc{s_{j}(\mathbb P^{n}(Y)_{+}) \ar[r]^-{\cong} & 
							s_{j}(\mathbb P ^{n+1}(Y)_{+}) 
							\ar[r]^-{\cong} & \cdots 
							 \ar[r]^-{\cong}& s_{j}(\mathbb P ^{\infty} (Y)_{+}) }
					\]
\end{thm}
\begin{proof}
	Let $k>n$, and consider the closed embedding induced by the diagram (\ref{eq.inf.proj.space})
	$\lambda ^{k}_{n}:\mathbb P^{n}(Y)\rightarrow \mathbb P^{k}(Y)$.  It is possible to choose a linear embedding
	$\mathbb P^{k-n-1}(Y)\rightarrow \mathbb P^{k}(Y)$ such that its open complement $U_{k,n}$ contains
	$\mathbb P^{n}(Y)$ and has the structure of a vector bundle over $\mathbb P^{n}(Y)$, with zero section
	$\sigma ^{k}_{n}$:
		\[ \xymatrix{ U_{k,n} \ar[r]^-{v^{k}_{n}} \ar@<-1ex>[d]& \mathbb P ^{k}(Y) & 
						\mathbb P ^{k-n-1}(Y) \ar[l]\\
						\mathbb P ^{n}(Y) \ar@<-1ex>[u]_-{\sigma ^{k}_{n}} \ar@/_1pc/[ur]_-{\lambda ^{k}_{n}} & &}
		\]
	By homotopy invariance $s_{<j}(\sigma ^{k}_{n})$ is an isomorphism in $\stablehomotopyX$
	for every integer $j$.
	On the other hand, if $j\leq n$, then $F_{0}(v ^{k}_{n})$ is a weak equivalence in $\jwbiratTspectraX$ since
	the codimension of its closed complement is $n+1$.  Thus,
	theorems \ref{thm.modelstructures-slicefiltration}(\ref{thm.modelstructures-slicefiltration.b}) and
	\ref{thm.orthogonal=wbirationalTspectra}
	imply that  if $j\leq n+1$, then $s_{<j}(v ^{k}_{n})$ is also an isomorphism in $\stablehomotopyX$.
	
	Therefore, $s_{<j}(\lambda ^{k}_{n})=s_{<j}(v^{k}_{n})\circ s_{<j}(\sigma ^{k}_{n})$ 
	is an isomorphism in $\stablehomotopyX$
	for $j\leq n+1$; and using lemma \ref{lemma.computingslices-via-birinvs} we conclude that the induced map
	on the slices $s_{j}(\lambda ^{k}_{n})$ is also an isomorphism for $j\leq n$.
	
	Finally, the result for $\mathbb P ^{\infty}(Y)$ follows directly from the fact that the slices commute
	with filtered homotopy colimits.
\end{proof}

	Let $H\mathbb Z$ denote Voevodsky's Eilenberg-MacLane spectrum (see \cite[section 6.1]{MR1648048})
	representing motivic cohomology in $\stablehomotopyX$.

\begin{cor}
		\label{cor.zero.slice.pinfty}
	Assume that the base scheme $X=\mathrm{Spec}\; k$, with $k$ a perfect field.  Then, in the following diagram 
	all the symmetric $T$-spectra are isomorphic to $H\mathbb Z$:
		\[	\xymatrix@C=1.5pc@R=0.5pc{H\mathbb Z \ar[r]^-{\cong} & s_{0}(\mathbb P^{0}(k)_{+}) \ar[r]^-{\cong} & 
							s_{0}(\mathbb P ^{1}(k)_{+}) 
							\ar[r]^-{\cong} & \cdots  \\
							 \cdots \ar[r]^-{\cong} & s_{0}(\mathbb P^{n}(k)_{+}) 
							\ar[r]^-{\cong} &\cdots 
							 \ar[r]^-{\cong}& s_{0}(\mathbb P ^{\infty} (k)_{+}) }
					\]
\end{cor}
\begin{proof}
	This follows immediately from theorem \ref{thm.slices.proj.space} together with the computation of
	Levine \cite[Thm. 10.5.1]{MR2365658} and 
	Voevodsky \cite{MR2101286} for the zero slice of the sphere spectrum.
\end{proof}
	
\begin{thm}
		\label{thm.slices-thom}
	Let $\iota _{U,Y}\in WB_{n}$, $\pi :V\rightarrow Y$ a vector bundle of rank $r$
	together with a trivialization $t:\pi ^{-1}(U)\rightarrow \mathbb A ^{r}_{U}$ of its restriction to $U$.
	Then for every integer $j\leq n$, there exists an isomorphism in $\stablehomotopyX$
	(see definition \ref{def.cofibres}(\ref{def.cofibres.b}))
		\[	s_{j}(Th(V))\cong S^{r}\wedge \Gm ^{r}\wedge s_{j-r}(Y_{+})
		\]
\end{thm}
\begin{proof}
	Let $Z\in Sm_{X}$ be the closed complement of $\iota _{U,Y}$.
	Consider the following diagram in $Sm_{X}$, where all the squares are cartesian
		\[	\xymatrix{\pi ^{-1}(Z)\cap (V\backslash \sigma _{0}(Y)) \ar[r] \ar[d] & V\backslash \sigma _{0}(Y) 
						\ar[d] & \pi ^{-1}(U) \cap (V\backslash \sigma _{0}(Y)) \ar[l]_-{\beta} \ar[d] \\
						\pi ^{-1}(Z) \ar[r] \ar[d]& V \ar[d]^-{\pi} & \pi ^{-1}(U) \ar[l]_-{\alpha} \ar[d] \\
						Z \ar[r] & Y & U \ar[l]^-{\iota _{U,Y}}}
		\]
	and let $\gamma:Th(\pi ^{-1}(U))\rightarrow Th(V)$ be the induced map between the corresponding Thom spaces.
	We observe that $\alpha, \beta$ also belong to $WB_{n}$; thus, if $j\leq n$ we conclude that
	$F_{0}(\iota _{U,Y}), F_{0}(\alpha ), F_{0}(\beta )$ are all weak equivalences in $\jwbiratTspectraX$.
	Therefore, theorems \ref{thm.modelstructures-slicefiltration}(\ref{thm.modelstructures-slicefiltration.b}) and
	\ref{thm.orthogonal=wbirationalTspectra}
	imply that  if $j\leq n+1$, then $s_{<j}(\iota _{U,Y})$, $s_{<j}(\alpha)$, $s_{<j}(\beta)$ 
	are isomorphisms in $\stablehomotopyX$.
	We claim that if $j\leq n+1$, then
		\[ s_{<j}(\gamma): s_{<j}(Th(\pi^{-1}(U))) \rightarrow s_{<j}(Th(V))
		\]
	is also an isomorphism in $\stablehomotopyX$.  In effect,
	by construction of the Thom spaces, we deduce that for any integer $j\in \mathbb Z$,
	the rows in the following commutative diagram in $\stablehomotopyX$ are in fact distinguished triangles
		\[	\xymatrix{s_{<j}((\pi ^{-1}(U)\cap (V\backslash \sigma _{0}(Y)))_{+}) \ar[r] \ar[d]_-{s_{<j}(\beta)} & 
			s_{<j}(\pi ^{-1}(U)_{+}) \ar[r] \ar[d]_-{s_{<j}(\alpha)} & s_{<j}(Th(\pi ^{-1}(U))) 
			\ar[d]_-{s_{<j}(\gamma)} \\
			s_{<j}((V\backslash \sigma _{0}(Y))_{+}) \ar[r] & s_{<j}(V_{+}) \ar[r]
			& s_{<j}(Th(V))}
		\]
	Since $s_{<j}(\alpha), s_{<j}(\beta)$ are isomorphisms in $\stablehomotopyX$ for $j\leq n+1$,
	we conclude that for $j\leq n+1$, $s_{<j}(\gamma)$ is also an isomorphism in $\stablehomotopyX$.
	
	Thus, lemma \ref{lemma.computingslices-via-birinvs} implies that for $j\leq n$, 
	$s_{j}(\iota _{U,Y}), s_{j}(\gamma )$ are isomorphisms in $\stablehomotopyX$.
	Now, we use the trivialization $t$ to obtain the following commutative diagram in $Sm_{X}$ where the
	rows are isomorphisms
		\[	\xymatrix{\mathbb A ^{r}_{U}\backslash U  \ar[d] & 
						& \pi ^{-1}(U) \cap (V\backslash \sigma _{0}(Y)) \ar[ll]_-{\tilde{t}}^-{\cong}\ar[d]\\
						\mathbb A ^{r}_{U}  \ar[dr]& & \pi ^{-1}(U) \ar[dl]^-{\pi _{U}} \ar[ll]_-{t}^-{\cong}\\
						& U &}
		\]
	The same argument as above, shows that for every integer $j\in \mathbb Z$, there is an isomorphism 
	in $\stablehomotopyX$
		\[ s_{j}(\bar{t}):s_{j}(Th(\pi ^{-1}(U)))
			\rightarrow s_{j}(Th(\mathbb A ^{r}_{U}))
		\]
	On the other hand, \cite[Prop. 2.17(2)]{MR1813224} implies that there is a weak equivalence
	$w:F_{0}(Th(\mathbb A ^{r}_{U})) \rightarrow S^{r}\wedge \Gm ^{r}\wedge F_{0}(U_{+})$
	in $\TspectraX$.  Thus, for $j\leq n$ there exist isomorphisms in $\stablehomotopyX$
		\[	\xymatrix{s_{j}(Th(\pi ^{-1}(U))) \ar[r]^-{s_{j}(\bar{t})} \ar[d]_-{s_{j}(\gamma)}& 
							s_{j}(Th(\mathbb A ^{r}_{U})) \ar[d]^-{s_{j}(w)}\\
							s_{j}(Th(V))  &s_{j}(S^{r}\wedge \Gm ^{r}\wedge U_{+})}
		\]
	However, there exists a canonical isomorphism in $\stablehomotopyX$
		\[	s_{j}(S^{r}\wedge \Gm ^{r}\wedge U_{+})\rightarrow S^{r}\wedge \Gm ^{r}\wedge s_{j-r}(U_{+})
		\]
	Finally, we conclude by using the isomorphism $s_{j-r}(\iota _{U,Y})$ (notice that if $j\leq n$ then certainly
	$j-r\leq n$, since $r\geq 0$).
\end{proof}

\begin{cor}
		\label{cor.slices-thom}
	Assume that the base scheme $X=\mathrm{Spec}\; k$, with $k$ a perfect field.
	Let $\iota _{U,Y}\in B_{n}$, $\pi :V\rightarrow Y$ a vector bundle of rank $r$
	together with a trivialization $t:\pi ^{-1}(U)\rightarrow \mathbb A ^{r}_{U}$ of its restriction to $U$.
	Then for every integer $j\leq n$, there exists an isomorphism in $\stablehomotopyX$
		\[	s_{j}(Th(V))\cong S^{r}\wedge \Gm ^{r}\wedge s_{j-r}(Y_{+})
		\]
\end{cor}
\begin{proof}
	Proposition \ref{prop.genericsmoothness-Quillenequiv} implies that $F_{0}(\iota _{U,Y})$ is a
	weak equivalence in $\jwbiratTspectraX$ for $j\leq n$.  Hence, the result follows using exactly
	the same argument as in theorem \ref{thm.slices-thom}.
\end{proof}

	Given a closed embedding $Z\rightarrow Y$ of smooth schemes over $X$, let $\blowup _{Z}Y$
	denote the blowup of $Y$ with center in $Z$.

\begin{thm}
		\label{thm.blowups}
	Let $\iota _{U,Y}\in WB_{n}$, and $j\in \mathbb Z$ an arbitrary integer.  
	Consider the following cartesian square in $Sm_{X}$
		\begin{equation}
				\label{eq.blowup}
			\begin{array}{c}
			\xymatrix{ D \ar[d]_-{q} \ar[r]^-{d}& \blowup _{Z}Y \ar[d]^-{p} & U \ar[l]_-{u} 
						\ar@{=}[d] \\
						 Z \ar[r]_-{i} & Y & U \ar[l]^-{\iota _{U,Y}}}
			\end{array}
		\end{equation}
	and let $q_{j}, d_{j}, p_{j}, i_{j}$ denote $s_{j}(q), s_{j}(d), s_{j}(p), s_{j}(i)$
	respectively.  Then the cartesian square (\ref{eq.blowup})
	induces the following distinguished triangle in $\stablehomotopyX$
		\begin{equation}
				\label{eq.blowuptriangle}
			\xymatrix{s_{j}(D_{+}) \ar[r]^-{\binom {-d_{j}} {q_{j}}} &
						s_{j}(\blowup _{Z}Y_{+})\oplus s_{j}(Z_{+}) 
						\ar[r]^-{(p_{j}, i_{j})} & s_{j}(Y_{+})}
		\end{equation}
	If $j\leq n$, then $s_{j}(\iota _{U,Y})$ is an isomorphism in $\stablehomotopyX$,
	and the following distinguished triangles in $\stablehomotopyX$ split
		\begin{align}
				\label{eq.split-blowuptriangle}
			\xymatrix{s_{j}(D_{+}) \ar@<1ex>[r]^-{\binom {-d_{j}} {q_{j}}} &
						s_{j}(\blowup _{Z}Y_{+})\oplus s_{j}(Z_{+}) 
						\ar@<1ex>[l] \ar@<1ex>[r]^-{(p_{j}, i_{j})} & 
						s_{j}(Y_{+}) \ar@<1ex>[l]^-{\binom {r_{j}} {0}}} \\
				\label{eq.second.split-blowuptriangle}
			\xymatrix{s_{j}(Y_{+}) \ar@<1ex>[r]^-{r_{j}} &
						s_{j}(\blowup _{Z}Y_{+})
						\ar@<1ex>[l]^-{p_{j}} \ar@<1ex>[r] & 
						s_{j}(Th(\mathcal O _{D}(1))) \ar@<1ex>[l]}
		\end{align}
	where $r_{j}=s_{j}(u)\circ (s_{j}(\iota _{U,Y}))^{-1}$, and
	$\mathcal O_{D}(1)$ denotes the canonical line bundle of the projective bundle
	$q:D\rightarrow Z$.
\end{thm}
\begin{proof}
	It follows from \cite[Prop. 2.29 and Rmk. 2.30]{MR1813224} that the following 
	square is homotopy cocartesian in $\mathcal M$
		\[	\xymatrix{ S^{1}\wedge D_{+} \ar[d]_-{id\wedge q} \ar[r]^-{id\wedge d}& 
					S^{1}\wedge \blowup _{Z}Y_{+} \ar[d]^-{id\wedge p} \\
						 S^{1}\wedge Z_{+} \ar[r]_-{id\wedge i} & S^{1}\wedge Y_{+}}
		\]
	Thus, we deduce that the following diagram is a distinguished triangle in $\stablehomotopyX$
		\[	\xymatrix{F_{0}(D_{+}) \ar[rr]^-{\binom {-F_{0}(d)} {F_{0}(q)}} &&
						F_{0}(\blowup _{Z}Y_{+})\oplus F_{0}(Z_{+})
						\ar[rr]^-{(F_{0}(p), F_{0}(i))} && F_{0}(Y_{+})}
		\]
	Since the slices $s_{j}$ are triangulated functors, it follows that diagram 
	(\ref{eq.blowuptriangle}) is a distinguished triangle in $\stablehomotopyX$.
	
	Now, we prove that $s_{j}(\iota _{U,Y})$ is an isomorphism for $j\leq n$.
	By lemma \ref{lemma.computingslices-via-birinvs}, it suffices to show that
	$s_{<j}(\iota _{U,Y})$ is an isomorphism in $\stablehomotopyX$ for $j\leq n+1$.
	But this follows directly from theorems \ref{thm.orthogonal=wbirationalTspectra} 
	and \ref{thm.modelstructures-slicefiltration}(\ref{thm.modelstructures-slicefiltration.b}) 
	since $F_{0}(\iota _{U,Y})$ is clearly a weak equivalence in $\jwbiratTspectraX$ for $j\leq n$.
	
	Thus, $r_{j}$ is well defined for $j\leq n$, and the following diagram shows
	that it gives a splitting for the distinguished triangle (\ref{eq.split-blowuptriangle})
		\begin{equation}
				\label{diagram.proof-split}
			\begin{array}{c}
			\xymatrix{s_{j}(U_{+}) \ar[rr]^-{s_{j}(u)} \ar@{=}[d] &&
					s_{j}(\blowup _{Z}Y_{+}) \ar[d]^-{p_{j}} \\
					s_{j}(U_{+}) \ar[rr]_-{s_{j}(\iota _{U,Y})} && s_{j}(Y_{+})}
			\end{array}
		\end{equation}
	Finally, since the normal bundle of the closed embedding $d:D\rightarrow \blowup _{Z}Y$
	is given by $\mathcal O _{D}(1)$, we deduce from the Morel-Voevodsky homotopy
	purity theorem (see \cite[Thm. 2.23]{MR1813224}) that the following diagram
	is a distinguished triangle in $\stablehomotopyX$
		\[	\xymatrix{s_{j}(U_{+}) \ar[rr]^-{s_{j}(u)} &&
						s_{j}(\blowup _{Z}Y_{+}) \ar[r] & 
						s_{j}(Th(\mathcal O _{D}(1)))}
		\]
	Combining this distinguished triangle with diagram (\ref{diagram.proof-split}) above,
	we conclude that diagram (\ref{eq.second.split-blowuptriangle})
	is a split distinguished triangle in $\stablehomotopyX$ for $j\leq n$.
\end{proof}

\end{section}


\section*{Acknowledgements}
	The author would like to thank Marc Levine for bringing to his attention the 
	connection between slices and birational cohomology theories, as well as for 
	all his stimulating comments and questions, and also thank Chuck Weibel
	for his interest in this work and several suggestions which helped to
	improve the exposition.


\bibliography{biblio_sboct}
\bibliographystyle{abbrv}

\end{document}